\documentclass[12pt,reqno]{amsart}

\usepackage[a4paper,margin=2.85cm]{geometry}
\usepackage{fontspec}
\usepackage{mathtools}
\usepackage{amsfonts}
\usepackage{unicode-math}
\usepackage[english]{babel}
\usepackage{microtype}
\usepackage{enumitem}
\usepackage{xurl}
\usepackage{xcolor}
\usepackage{tikz}
\usepackage{flafter}
\usepackage{aliascnt}
\usetikzlibrary{arrows.meta,positioning,calc}
\usepackage[colorlinks=true,linkcolor=blue!45!black,citecolor=green!35!black,urlcolor=red!45!black]{hyperref}
\hypersetup{
 pdftitle={A Gap Theorem for the Mobius Cross Energy Near the Hopf Link},
 pdfauthor={Chuanhuan Li and Ronggang Li},
 pdfsubject={0905v2 manuscript revision}}
\usepackage[nameinlink,noabbrev]{cleveref}

\allowdisplaybreaks[3]
\numberwithin{equation}{section}
\setlist[itemize]{leftmargin=1.8em,itemsep=.1em,topsep=.2em}
\setlist[enumerate]{leftmargin=2.1em,itemsep=.1em,topsep=.2em}

\theoremstyle{plain}
\newtheorem{theorem}{Theorem}[section]
\newaliascnt{proposition}{theorem}
\newtheorem{proposition}[proposition]{Proposition}
\aliascntresetthe{proposition}
\newaliascnt{lemma}{theorem}
\newtheorem{lemma}[lemma]{Lemma}
\aliascntresetthe{lemma}
\newaliascnt{corollary}{theorem}

\aliascntresetthe{corollary}
\theoremstyle{definition}
\newaliascnt{definition}{theorem}
\newtheorem{definition}[definition]{Definition}
\aliascntresetthe{definition}
\theoremstyle{remark}
\newaliascnt{remark}{theorem}
\newtheorem{remark}[remark]{Remark}
\aliascntresetthe{remark}

\crefname{theorem}{theorem}{theorems}
\crefname{proposition}{proposition}{propositions}
\crefname{lemma}{lemma}{lemmas}
\crefname{corollary}{corollary}{corollaries}
\crefname{remark}{remark}{remarks}
\crefname{section}{section}{sections}

\newcommand{\Sph}{\mathbb S}
\newcommand{\R}{\mathbb R}
\newcommand{\Z}{\mathbb Z}
\newcommand{\sA}{\mathscr A}

\newcommand{\lk}{\operatorname{lk}}
\newcommand{\Mob}{\operatorname{M\ddot{o}b}}
\newcommand{\Diff}{\operatorname{Diff}}
\newcommand{\dist}{\operatorname{dist}}
\newcommand{\diam}{\operatorname{diam}}
\newcommand{\Span}{\operatorname{span}}
\newcommand{\Hom}{\operatorname{Hom}}
\newcommand{\Id}{\operatorname{Id}}
\newcommand{\tr}{\operatorname{tr}}

\newcommand{\grad}{\operatorname{grad}}

\newcommand{\dd}{\,\mathrm d}
\newcommand{\eps}{\varepsilon}

\newcommand{\inner}[2]{\left\langle #1,#2\right\rangle}
\newcommand{\Hopf}{\mathsf H}
\newcommand{\KH}{\mathcal K_{\Hopf}}

\newcommand{\EL}{\mathcal G}
\newcommand{\Graph}{\operatorname{Graph}}

\newcommand{\LH}{\mathcal L_{\Hopf}}
\newcommand{\ngraph}{\operatorname{ng}_{\Hopf}}
\newcommand{\Torus}{\mathbb T^2}
\newcommand{\Lin}{\operatorname{Lin}}

\title[A gap theorem for the M\"obius cross energy]
{A Gap Theorem for the M\"obius Cross Energy Near the Hopf Link}

\author{Chuanhuan Li}
\address{Shanghai Institute for Mathematics and Interdisciplinary Sciences (SIMIS), Shanghai 200433, China \newline
${\quad}$ Research Institute of Intelligent Complex Systems, Fudan University, Shanghai 200433, China}

\email{chli@simis.cn}

\author{Ronggang Li$^\ast$}
\address{School of Mathematics and Statistics, Nanjing University of Information Science and Technology, Nanjing, China}
\email{003705@nuist.edu.cn}

\thanks{$^{\ast}$Corresponding author}

\subjclass[2020]{Primary 53A31, 58E05; Secondary 57K10, 49Q10, }
\keywords{M\"obius cross energy, Hopf link, critical-value gap} 

\begin{document}

\begin{abstract}
We prove that the critical value $2\pi^2$ is isolated for the M\"obius cross energy of two-component links in the round three-sphere $\mathbb{S}^{3}$. Specifically, there exists $\eps_0>0$ such that any non‑split, regular $H^2$ pair of curves with disjoint images, having vanishing first variation and energy at most $2\pi^2+\eps_0$, must lie in the M\"obius‑reparametrization orbit of the standard Hopf link. 
\end{abstract}

\maketitle

\section{Introduction and main theorem}\label{sec:intro}

Set \(\Sph^1=\R/(2\pi\Z)\), and equip
\(\Sph^3=\{z\in\R^4:|z|=1\}\) with the round metric \(g_0\).
For a pair \(\Gamma=(\gamma_1,\gamma_2)\) of rectifiable
closed curves with disjoint component images, the M\"obius cross
energy is
\begin{equation}\label{eq:energy}
 E(\Gamma)
 =\int_{\gamma_1}\int_{\gamma_2}
 \frac{\dd\ell_1(x)\,\dd\ell_2(y)}{|x-y|^2},
\end{equation}
where \(|x-y|\) is the Euclidean chordal distance in \(\R^4\).
Integration along each parametrized curve counts multiplicity.
O'Hara~\cite{Ohara1991} introduced the related knot energy. Freedman, He,
and Wang~\cite{FHW1994} studied the cross interaction
\eqref{eq:energy}, which is invariant under M\"obius transformations.

A pair of disjoint compact connected sets is called \emph{non-split}
if no smooth embedded two dimensional sphere disjoint from them separates the two sets. Hopf links are non-split, and the standard Hopf link is
\begin{equation}\label{eq:Hopf}
 \Hopf=(x,y),\qquad
 x(s)=(\cos s,\sin s,0,0),\qquad
 y(t)=(0,0,\cos t,\sin t).
\end{equation}
Its components are unit-speed great circles in orthogonal
\(2\)-planes.  Since \(|x(s)-y(t)|^2=2\),
$$
 E(\Hopf)=2\pi^2.
$$
Freedman, He, and Wang~\cite{FHW1994} conjectured that this is the minimum cross
energy among all non-split two-component links.
An important step was taken by He~\cite{He2002}, who proved that a
global minimizer exists and that every such minimizer is ambiently
isotopic to the Hopf link.  Agol, Marques, and Neves subsequently
proved the sharp lower bound and classified all equality
configurations~\cite{AMN2016}.

\begin{theorem}[Agol--Marques--Neves]\label{thm:AMN}
Every non-split two-component rectifiable link in \(\Sph^3\)
satisfies
\[
 E(\Gamma)\ge2\pi^2.
\]
Equality holds precisely when a conformal transformation carries the
component images to the standard Hopf circles, each counted with
multiplicity one.
\end{theorem}

For a pair of curves \(\widehat\Gamma=(\widehat\gamma_1,\widehat\gamma_2)\) in $\mathbb{R}^3$, the Gauss linking number is defined by 
$$
\lk(\widehat\Gamma)
=\frac{1}{4\pi}
\int_{\Sph^1}\int_{\Sph^1}
\frac{
\det\bigl(\widehat\gamma_1'(s),\widehat\gamma_2'(t),
\widehat\gamma_1(s)-\widehat\gamma_2(t)\bigr)}
{|\widehat\gamma_1(s)-\widehat\gamma_2(t)|^3}
\,\dd s\,\dd t.
$$
It is conformal invariant, therefore the linking number of a pair of curves $\Gamma$ in $\Sph^3$ is defined to be that of its stereographic projection in $\mathbb{R}^3$. A nonzero linking number implies non-splitness, but the converse fails, as illustrated by the Whitehead link. A direct calculus implies
$$
E(\Gamma)\geq 4\pi|\lk(\Gamma)|.
$$
Allowing a choice of component orientations, He ~\cite{He2002} shows that a pair of non-split curves is link-homotopic to the Hopf link provided $E(\Gamma)<8\pi$ .

As a continuation of the study of He, Agol, Marques, and Neves \cite{He2002, AMN2016},and motivated by Mondino and Nguyen \cite{MN14}, we study critical points rather than minimizers of the cross energy.  The equality case in \Cref{thm:AMN} characterizes the minimizers, but does not rule out non-minimizing critical pairs with energy arbitrarily close to \(2\pi^2\). Our main result shows that such critical pairs do not exist.  

To state the result precisely, we fix the following notation.

\begin{definition} A closed curve $\gamma $ is \emph{regular} if \(|\gamma'|>0\) everywhere, we fix the extrinsic Sobolev space
\[
 H^2(\Sph^1,\Sph^3)
 =\left\{\gamma\in H^2(\Sph^1,\R^4):
          |\gamma|=1\ \text{a.e.}\right\}.
\]

\begin{enumerate}
\renewcommand{\labelenumi}{\textup{(\roman{enumi})}}

\item The space of pairs we study is
$$
\sA=\left\{(\gamma_1,\gamma_2):
 \begin{array}{l}
 \gamma_1 , \gamma_2\in H^2(\Sph^1,\Sph^3) \text{ are regular curves}  \text{ s.t. }
 \gamma_1(\Sph^1)\cap\gamma_2(\Sph^1)=\varnothing
 \end{array}\right\}.
 $$

\item A pair is \emph{weakly critical} if the first variation \(\delta E_\Gamma=0\) on every \(H^2\) tangent variation.  

\item Let \(\Diff^2(\Sph^1)\) denote the \(H^2\) circle maps that are
\(C^1\) diffeomorphisms, allowing both orientations, and set
\begin{equation}\label{eq:G}
 {\rm G}
 =\Mob^+(\Sph^3)\times\Diff^2(\Sph^1)\times\Diff^2(\Sph^1).
\end{equation}
We write \({\rm G}\cdot\Hopf\) for the set of pairs obtained from
\(\Hopf\) by a common orientation-preserving M\"obius
transformation and separate reparametrizations.
\end{enumerate}
\end{definition}

We can now state our main gap theorem.
\begin{theorem}\label{thm:main}
There exists \(\eps_0>0\) such that every non-split, weakly
critical pair \(\Gamma\in\sA\) satisfying
\[
 E(\Gamma)\le2\pi^2+\eps_0
\]
belongs to \({\rm G}\cdot\Hopf\).
In particular, \(E(\Gamma)=2\pi^2\).
\end{theorem}

\begin{remark}
For fixed $\gamma_2$, set
\[
 \rho_2(x)=\int_{\gamma_2}\frac{\dd\ell_2(y)}{|x-y|^2}.
\]
Then $E(\gamma_1,\gamma_2)$ is the length of $\gamma_1$ in the conformal
metric $\rho_2^2g_0$ on $\Sph^3\setminus\gamma_2$; see~\cite{FHW1994}.
Thus a critical pair consists of unparametrized closed geodesics in the
metrics generated by the opposite components. We derive the corresponding
equation in \Cref{sec:first}.
\end{remark}

Exact criticality is essential here, since non-Möbius perturbations of the Hopf link can have energies arbitrarily close to $2\pi^2$. Although we work with regular $H^2$ pairs, the critical equation implies smoothness after constant-speed reparametrization.

${}$

\paragraph{\bf Outline of the paper.}
Section \ref{sec:first} derives the first variation and the Euler-Lagrange equation for weakly critical pairs. Section \ref{sec:global} establishes a M\"obius normalization, proves regularity of weakly critical pairs, and develops the compactness theory near the minimum energy. In Section \ref{sec:Hopf-Hessian}, we compute the second variation at the Hopf link, identify its kernel, and establish transverse coercivity. Section \ref{sec:local} develops the nonlinear estimates and proves the renormalization lemma. Finally, in Section \ref{sec:gap}, we combine the transverse analysis of Section \ref{sec:Hopf-Hessian} with the renormalization argument of Section \ref{sec:local} to prove the local gap theorem, and then use the compactness result of Section \ref{sec:global} to complete the proof of the main theorem.

${}$

\paragraph{\textbf{Acknowledgments.}}
The authors grateful to Professor Jian Ge for useful discussions. They also thank Xingwang Xu, Xingdong Tang and Jie Zhou for their interest in this problem. The first author is supported by the China Postdoctoral Science Foundation (Grant No.~2026M793350).  The second author is supported by the Open Project Program of the Key Laboratory of Mathematics and Complex System at Beijing Normal University (Grant No.~202501). The principal arguments and the core of the paper were developed by the second author while visiting SIMIS.

${}$

\paragraph{\bf AI assistance declaration:}
The main ideas and most proofs are solely the work of the authors. ChatGPT (version 5.6sol) was used only to assist in proving Proposition \ref{prop:gauge} and to aid in manuscript preparation. The authors have verified all contributions and assume complete accountability for the paper.

\section{First variation and the Euler–Lagrange equation}\label{sec:first}
In this section, we establish the basic preliminaries for the M\"obius cross energy and derive its first variation and Euler–Lagrange equation.

We denote $D$ the Euclidean connection on $\mathbb{R}^{4}$ and $\nabla$ the Levi-Civita connection of $g_0$ on $\mathbb{S}^{3}$. Let $X,Y$ be tangent vector fields on $\mathbb{S}^{3}$, then for each $x\in \mathbb{S}^{3}$, the Gauss formula gives
\begin{align}\label{eq:Gauss}
    D_{X}Y=\nabla_{X}Y+ II(X,Y),\quad II(X,Y)=-\langle X,Y\rangle_{g_{0}} x,
\end{align}
where $II$ is the second fundamental form of $\mathbb{S}^{3}\subset\mathbb{R}^{4}$.

Thus, for a unit-speed curve, $\gamma''=\kappa-\gamma$, where
$\kappa=\nabla_{\gamma'}{\gamma'}$ is its curvature vector in $\Sph^3$. 

For a variation $\Gamma_\tau$ with initial velocity $\xi$, we write
$\delta E_\Gamma[\xi]=\left.\frac{\mathrm d}{\mathrm d\tau}E(\Gamma_\tau)\right|_{\tau=0}$ and  $\delta^2$ for the second differential in a fixed chart. Ordinary coordinate derivatives are denoted by $\partial$ or $\mathrm d$. The parameter torus is $\Torus=\Sph^1_s\times\Sph^1_t$.  Throughout the paper,
\begin{align}
 \iint F(s,t)\dd s\dd t
 :=\int_0^{2\pi}\int_0^{2\pi}F(s,t)\dd s\dd t \notag
\end{align}
always refers to this fixed torus with the product parameter measure.

\begin{lemma}\label{lem:constant-speed}
Every regular closed $H^2$ curve admits a constant-speed parametrization obtained by an orientation-preserving $H^2$ circle diffeomorphism.  Separate such reparametrizations preserve the cross energy, non-splitness, and weak criticality.
\end{lemma}

\begin{proof}
For a lift of the parameter to $[0,2\pi]$, set
\[
 \theta(s)=\frac{2\pi}{L(\gamma)}\int_0^s|\gamma'(r)|\dd r, \qquad s\in [0,2\pi].
\]
Since $\gamma'\in H^1\hookrightarrow C^0$ and $\gamma$ is a regular closed curve,
composition with the Euclidean norm gives $|\gamma'|\in H^1$. Consequently $\theta\in H^2$, and $\theta'=2\pi|\gamma'|/L(\gamma)$ is continuous and bounded below by a positive constant.  The identity
$\theta(s+2\pi)=\theta(s)+2\pi$ makes $\theta$ a degree-one $C^1$ circle diffeomorphism.  Denote its inverse by $\psi=\theta^{-1}$,  
change of variables gives $\psi\in H^2$.
Thus we can define $\bar\gamma:=\gamma\circ\psi$ satisfies $\bar\gamma\in H^2$ and
$|\bar\gamma'|=L(\gamma)/(2\pi)$.

Changing variables separately in each of the two integrals proves the energy invariance, since the curve images remain unchanged, the non‑split property is also preserved. 
Composition by each fixed $H^2$ diffeomorphism preserves $H^2$ tangent fields.  Pulling back a variation and differentiating energy invariance proves preservation of weak criticality in both directions.  A fixed smooth M\"obius transformation preserves weak criticality by the same argument.
\end{proof}

Next, we derive the first variation in the spherical model.
Let $\Gamma=(\gamma_{1},\gamma_{2})\in\sA$ be a regular $H^{2}$ pair with disjoint component images in $\mathbf{S}^{3}$. 
Write
$T_i=\partial_{\ell_i}\gamma_i$ and $\kappa_i=\nabla_{T_i}T_i$.
The projection onto the curve-normal plane inside $T_{\gamma_i}\Sph^3$ is
\begin{equation}\label{eq:normal-projection}
 P_{N_i}=\Id-\gamma_i\otimes\gamma_i-T_i\otimes T_i.
\end{equation}
Define the positive interaction potentials by
\begin{equation}\label{eq:potentials}
 \rho_1(y)=\int_{\gamma_1}\frac{\dd\ell_1(x)}{|x-y|^2},\qquad
 \rho_2(x)=\int_{\gamma_2}\frac{\dd\ell_2(y)}{|x-y|^2}.
\end{equation}
For $r=|x-y|$ and $K(x,y)=r^{-2}$, 
spherical differentiation yields
\begin{equation}\label{eq:kernel-gradient}
 \grad_xK=\frac{2(y-\inner{x}{y}x)}{r^4},\qquad
 \grad_yK=\frac{2(x-\inner{x}{y}y)}{r^4}.
\end{equation}
In particular, $P_{N_1}\grad_xK=2P_{N_1}y/r^4$, and similarly for the second component.  
\begin{proposition}\label{prop:first-variation}
For every $\Gamma\in\sA$, define the $L^2$ normal fields
\[
 F_1=P_{N_1}\grad\rho_2-\rho_2\kappa_1,\qquad
 F_2=P_{N_2}\grad\rho_1-\rho_1\kappa_2.
\]
For every $H^1$ sphere-tangent field $\xi=(\xi_1,\xi_2)$, the extended first variation is
\begin{equation}\label{eq:first-variation}
 \delta E_\Gamma[\xi]
 =\int_{\gamma_1}\inner{F_1}{\xi_1}\dd\ell_1
 +\int_{\gamma_2}\inner{F_2}{\xi_2}\dd\ell_2.
\end{equation}
Consequently, weak criticality is equivalent to the Euler–Lagrange equations
\begin{equation}\label{eq:EL}
 \kappa_1=P_{N_1}\grad\log\rho_2,\qquad
 \kappa_2=P_{N_2}\grad\log\rho_1
\end{equation}
almost everywhere.  

\end{proposition}

\begin{proof}  
We write 
$$E(\Gamma)=\int_{\gamma_{i}}\rho_{j}d\ell_{i},\quad i.j=1,2,\quad i\neq j.$$
For an \(H^2\) sphere-tangent variation field \(\xi_i\), differentiation and integration by parts along the closed curve give
\begin{align}
    \delta E_\Gamma[\xi_{i}]&=\int_{\gamma_i}\left(\inner{\grad\rho_{j}}{\xi_{i}}+\rho_{j}\inner{T_i}{\nabla_{T_i}\xi_{i}}\right)
 \dd\ell_i\notag\\
 &=\int_{\gamma_i}\left(\inner{\grad\rho_{j}-\nabla_{T_i}(\rho_{j}T_{i})}{\xi_{i}}\right)
 \dd\ell_i\notag\\
 &=\int_{\gamma_i}\inner{F_{i}}{\xi_{i}}\dd\ell_i\notag,
\end{align}
where the last equality uses
$$\nabla_{T_i}(\rho_{j}T_{i})=\inner{\grad \rho_{j}}{T_i}T_i+\rho_{j}\kappa_{i}.$$
In arclength coordinates, \(T_i\in H^1\) and \(\kappa_i\in L^2\), so the integration by parts is valid in the Sobolev sense. Since $F_i\in L^2$, the right-hand side extends continuously to \(H^1\) sphere-tangent test fields. Summing the two partial variations proves \eqref{eq:first-variation}.

Given an arbitrary smooth periodic $\R^4$-valued field $\varphi$,
take an admissible sphere-tangent field
\[
 \xi_i=(\Id-\gamma_i\otimes\gamma_i)\varphi\in H^2.
\]
 Since $F_i\perp\gamma_i$, weak criticality gives $\int\inner{F_i}{\varphi}\dd\ell_i=0$.
The speed is bounded away from zero implies $F_i=0$ almost everywhere.  

Conversely, $F_i=0$ makes every first variation vanish.  Dividing by the positive potentials proves \eqref{eq:EL}.
\end{proof}
\begin{remark}
He \cite[Lemma~3.1, equations~(3-2)-(3-3)]{He2002} derived the first variation for smooth links and smooth variation fields. We repeat the computation here in order to justify the formula for regular $H^2$ links and $H^{1}$ variation fields, which obtain the weak Euler–Lagrange equation used below.
\end{remark}

\section{Global normalization and compactness of critical pairs}\label{sec:global}
In this section, we establish a global M\"obius gauge for non-split pairs in \(\sA\). Combined with the critical equation, this yields regularity and uniform derivative estimates for weakly critical pairs. We then prove a compactness theorem for critical sequences approaching the Hopf energy, which will be used in the proof of the gap theorem.

We first give a global M\"obius gauge. Let $N=(0,0,0,1)$ and let $\sigma:\Sph^3\setminus\{N\}\to\R^3$ be stereographic projection.  Its inverse satisfies
\begin{equation}\label{eq:stereo}
 \sigma^{-1}(z)=\left(\frac{2z}{1+|z|^2},\frac{|z|^2-1}{1+|z|^2}\right),
 \qquad
 |\sigma^{-1}(z)-\sigma^{-1}(w)|
 =\frac{2|z-w|}{\sqrt{(1+|z|^2)(1+|w|^2)}}.
\end{equation}

\begin{proposition}\label{prop:gauge}
Fix \(0<\Lambda<\infty\). Every non-split pair
\(\Gamma=(\gamma_1,\gamma_2)\in\sA\) satisfying
$E(\Gamma)\le \Lambda$ has an orientation-preserving M\"obius image for which
\begin{equation}\label{eq:explicit-length-bounds}
 \frac{2\pi}{3}
 \le L(\gamma_i)
 \le \frac{6\Lambda}{\pi},
 \qquad i=1,2,
\end{equation}
and
\begin{equation}\label{eq:gauge-bounds}
\dist_{\R^4}(\gamma_1,\gamma_2) \ge  d_\Lambda>0.
\end{equation}
The constant \( d_\Lambda\) depends only on \(\Lambda\).
\end{proposition}

\begin{proof}
Send a point of $\gamma_1$ to the north pole $N$, and then apply a translation and dilation in the stereographic chart $\sigma$ so that
$$ N\in\gamma_1,\qquad 0\in\sigma(\gamma_2),\qquad \operatorname{diam}_{\mathbb R^3}\sigma(\gamma_2)=1. $$
Thus $\sigma(\gamma_2)\subset\overline B_1(0)$. Non-splitness forces $\gamma_1$ to meet $\sigma^{-1}(\overline B_1(0))$; otherwise the inverse image of a slightly larger round sphere would separate the components.

We also use the elementary estimate
\begin{equation}\label{eq:closed-curve-diameter}
 L(\eta) \ge 2\,\diam_{\Sph^3}\eta(\Sph^1)
\end{equation}
for every closed rectifiable curve
\(\eta:\Sph^1\to\Sph^3\) by considering the two arcs joining a diameter-realizing pair.
The first component contains $N$ and a point  $q\in\gamma_1(\Sph^1)$ with $|\sigma(q)|\le1$, so its spherical diameter is at least $\pi/2$.
Applying \eqref{eq:closed-curve-diameter} gives
\begin{equation}\label{eq:L1-lower}
 L(\gamma_1)\ge\pi.
\end{equation}

The compact set \(\sigma(\gamma_2(\Sph^1))\) realizes its diameter, so there exist \(z,w\) in this set such that
\[
 |z-w|=1.
\]
Note that \(|z|,|w|\le1\).
If \(d\) denotes their spherical distance, then formula~\eqref{eq:stereo} gives
\[
 2\sin\frac d2
 =\bigl|\sigma^{-1}(z)-\sigma^{-1}(w)\bigr|=
 \frac{2|z-w|}{\sqrt{1+|z|^2}\sqrt{1+|w|^2}}\ge1,
\]
and hence \(d\ge\pi/3\). Another application of
\eqref{eq:closed-curve-diameter} yields
\begin{equation}\label{eq:L2-lower}
 L(\gamma_2)\ge\frac{2\pi}{3}.
\end{equation}

Since \(|x-y|^2\le4\) for \(x,y\in\Sph^3\),
\begin{equation}\label{eq:energy-product-length}
 E(\Gamma)
 =
 \int_{\gamma_1}\int_{\gamma_2}
 \frac{\dd\ell_1\,\dd\ell_2}{|x-y|^2}
 \ge
 \frac14L(\gamma_1)L(\gamma_2).
\end{equation}
Combining \(E(\Gamma)\le\Lambda\) with
\eqref{eq:L1-lower}--\eqref{eq:L2-lower} gives the stronger individual estimates
\[
 L(\gamma_1)
 \le
 \frac{4\Lambda}{L(\gamma_2)}
 \le
 \frac{6\Lambda}{\pi},
 \qquad
 L(\gamma_2)
 \le
 \frac{4\Lambda}{L(\gamma_1)}
 \le
 \frac{4\Lambda}{\pi}.
\]
Together with \eqref{eq:L1-lower}--\eqref{eq:L2-lower}, these imply the symmetric bounds
\eqref{eq:explicit-length-bounds}.

To obtain a uniform lower bound for the distance between the two components, we set
\[
  d
 :=
 \dist_{\R^4}(\gamma_1,\gamma_2)
 =
 \min_{x\in\gamma_1(\Sph^1),\,y\in\gamma_2(\Sph^1)}
 |x-y|.
\]
Using unit-speed arclength coordinates on $\R/L(\gamma_i)\Z$
with the two origins chosen at a closest pair, so that
\[
 |\gamma_1(0)-\gamma_2(0)|= d.
\]
Put $\ell:=\frac{2\pi}{3}$. The lower length bounds ensure that both curves contain arclength intervals \([0,\ell]\). For \(0\le s,t\le\ell\),
the triangle inequality and the fact that chordal distance is bounded above by arclength give
\begin{align*}
 |\gamma_1(s)-\gamma_2(t)|
 &\le
 |\gamma_1(s)-\gamma_1(0)|
 +|\gamma_1(0)-\gamma_2(0)|
 +|\gamma_2(0)-\gamma_2(t)|  \\
 &\le s+ d+t.
\end{align*}
Restricting the energy integral to this parameter rectangle, we obtain
\begin{align}
 \Lambda
 \ge E(\Gamma)
 &\ge
 \int_0^\ell\int_0^\ell
 \frac{\dd s\,\dd t}{( d+s+t)^2} =
 \log\frac{( d+\ell)^2}
 { d( d+2\ell)}.\notag
\end{align}
Solving this inequality gives an explicit choice:
\begin{align}
  d\ge d_\Lambda
 :=\frac{2\pi}{3}\left((1-e^{-\Lambda})^{-1/2}-1\right)>0.\notag
\end{align}
Thus we complete the proof.
\end{proof}


\begin{remark}

Ge~\cite{Ge24} proved that any two disjoint
linked Jordan curves \(\gamma_{1},\gamma_{2}\subset\Sph^3\) satisfy
$$\operatorname{dist}_{\Sph^3}(\gamma_{1},\gamma_{2})\le\frac{\pi}{2},$$
with equality precisely for two great circles in orthogonal $2$-planes.
\end{remark}

From the Euler–Lagrange equation, we obtain regularity and uniform derivative estimates for weakly critical pairs.

\begin{lemma}\label{lem:potential}
Assume $\dist(\gamma_1,\gamma_2)\ge d>0$. Then, for every integer $m\ge0$, $i,j=\{1,2\}$, $i\neq j$,
\[
|{\nabla^m\rho_j}|_{L^\infty(\gamma_i)}
 \le C_mL(\gamma_j)d^{-m-2},
 \qquad
 \rho_j\ge\frac{L(\gamma_j)}4.
\]
For $m\ge1$,
\begin{equation}\label{eq:logrho-estimates}
 |{\nabla^m\log\rho_j}|_{L^\infty(\gamma_i)}\le C_m(d)
\end{equation}
with $C_m(d)$ independent of  $L(\gamma_j)$. 
\end{lemma}

\begin{proof}
We only prove the case $i=1,j=2$.
Repeated Euclidean differentiation of $K(x,y)=|x-y|^{-2}$ yields a finite sum of tensors 
\[
D_x^mK(x,y)=|x-y|^{-m-2}P_m\left(\frac{x-y}{|x-y|}\right),
\]
where $P_m$ is a universal tensor-valued polynomial.  
The Gauss formula \eqref{eq:Gauss} gives
\begin{align}
 |\nabla_x^mK(x,y)|\le C_m|x-y|^{-m-2}\le C_md^{-m-2}.\notag
\end{align}
Arclength parametrization and dominated convergence justify differentiation under the integral:
\[
 \nabla^m\rho_2(x)
 =\int_{\gamma_2}\nabla_x^mK(x,y)\dd\ell_2(y).
\]
This proves the first estimate.  Since $|x-y|^2\le4$, the kernel is at least $1/4$, giving
the lower bound.

For the logarithmic estimates, write
\[
 \rho_2=L(\gamma_{2})\bar\rho_2,
 \qquad
 \bar\rho_2(x)=\frac1{L(\gamma_{2})}\int_{\gamma_2}K(x,y)\dd\ell_2(y).
\]
The function $\bar\rho_2$ is bounded below by $1/4$, and all of its ambient derivatives are bounded by constants depending only on $m$ and $d$.  Since $\log\rho_2=\log L(\gamma_{2})+\log\bar\rho_2$, every derivative of positive order is independent of the constant factor $L(\gamma_{2})$.  The ordinary chain rule for $\log$ proves \eqref{eq:logrho-estimates}.
\end{proof}

\begin{proposition}\label{prop:Cm}
Every constant-speed weakly critical pair in $\sA$ is smooth. Moreover, let $\Gamma_n=(\gamma_{1,n},\gamma_{2,n})$ be sequences of constant-speed weakly critical pairs, under the bounds
\[
 0<L_-\le L(\gamma_{i,n})\le L_+<\infty,
 \qquad
 \dist(\gamma_{1,n},\gamma_{2,n})\ge d_0>0.
\]
one has $|\gamma_{i,n}|_{C^{m}}\leq C_{m}(L_{-},L_{+},d)$ for every $m\ge0$. 
\end{proposition}

\begin{proof}
By \Cref{lem:constant-speed}, take
\[
 |\gamma_i'|\equiv a_i=\frac{L_i}{2\pi}>0,
 \qquad
 T_i=\frac{\gamma_i'}{a_i}.
\]
Thus the curvature and $\gamma''$ satisfy
\[
\gamma_i''=a_i^2(\kappa_i-\gamma_i).
\]
Combining this identity with \eqref{eq:EL} gives the following  Euler-Lagrange equation:
\begin{equation}\label{eq:critical-ODE}
 \gamma''_{i}=a_{i}^2\left(P_{N_i}G_{j}(\gamma_{i})-\gamma_{i}\right),
 \qquad
 G_{j}=\grad\log\rho_j,
 \qquad j\ne i.
\end{equation}
Here $P_{N_i}$ is the projection onto the normal space. 

By Lemma \ref{lem:potential}, $G_{j}$ extends smoothly near $\gamma_{i}$. Initially $H^2\hookrightarrow C^{1,1/2}$, so $\gamma_{i}\in C^{1,1/2}$, which implies $\gamma_i\in C^{2,1/2}$.  
Repeating this argument yields $\gamma_i\in C^\infty$.

Under the stated bounds, $|\gamma_{i,n}|=1$ and $|\gamma_{i,n}'|=a_{i,n}$ give a uniform $C^1$ bounded.  Since \Cref{lem:potential} gives common bounds for every ambient derivative of $G_{j,n}$ in a fixed neighborhood of the opposite component, \eqref{eq:critical-ODE} gives the $C^2$ bound.
The higher-order estimates follow inductively. Indeed, if
$$
|\gamma_{i,n}|_{C^m}\le C_m
$$
for some $m\ge2$, then differentiating \eqref{eq:critical-ODE} $m-1$ times expresses $\gamma_{i,n}^{(m+1)}$ as a finite sum involving only derivatives of $\gamma_{i,n}$ of order at most $m$ and ambient derivatives of $G_{j,n}$ of order at most $m-1$. All these terms are uniformly bounded by the induction hypothesis and \eqref{eq:logrho-estimates}. Hence
$$
|\gamma_{i,n}|_{C^{m+1}}\le C_{m+1}(L_{-},L_{+},d).
$$
Induction proves the result.
\end{proof}


\begin{lemma}\label{lem:nonsplit-closed}
Suppose that non-split curve pairs $\Gamma_n$ converge uniformly to a curve pair $\Gamma_\infty$ whose component images are disjoint.  Then $\Gamma_\infty$ is non-split.
\end{lemma}

\begin{proof}
If a smooth embedded sphere $S$ separated the two limiting images, compactness would give a positive distance from $S$ to each image.  Uniform convergence then places each approximating image, for large $n$, on the same side of $S$ as its limit.  The same fixed sphere would split $\Gamma_n$, a contradiction.
\end{proof}

Combining the preceding normalization, regularity, and uniform estimates, we obtain the following compactness theorem for sequences of weakly critical pairs.
\begin{theorem}\label{prop:threshold}
Let $\Gamma_n\in\sA$ be non-split and weakly critical, with
$E(\Gamma_n)\to2\pi^2$ as $n\rightarrow\infty$. Up to M\"obius transformations and separate reparametrizations of the components, a subsequence converges to $\Hopf$ in $C^m$ for every $m$.
\end{theorem}

\begin{proof}
Discarding finitely many terms, fix $\Lambda$ with $E(\Gamma_n)\le\Lambda$. Apply \Cref{prop:gauge} and \Cref{lem:constant-speed} to obtain two-sided length bounds, a uniform positive chordal separation, and constant-speed parametrizations. The curves are smooth and have uniform $C^m$ bounds for every $m$ by \Cref{prop:Cm}.  From Arzel\`a--Ascoli theorem and a diagonal subsequence, we get a limit $\Gamma_\infty$ in every $C^m$ norm.

The lower speed bound persists under $C^1$ convergence, and the separation bound persists under uniform convergence.  Thus $\Gamma_\infty\in\sA$. It is non-split by \Cref{lem:nonsplit-closed}.  On the fixed parameter torus the energy densities converge uniformly, since their denominators stay away from zero.  Hence
\[
 E(\Gamma_\infty)=\lim_{n\rightarrow\infty} E(\Gamma_n)=2\pi^2.
\]
The equality case of Theorem 1.1 implies that after a conformal map $\Phi\in\Mob^+(\Sph^3)$ (labeling and orientation fixed as explained after \eqref{eq:G}), each $\Phi\circ\gamma_{i,\infty}$ is a regular covering of a standard Hopf circle of some degree $m_{i}\neq 0$,  and
$$2\pi^2= E(\Gamma_\infty)= E(\Phi\Gamma_\infty)=2\pi^2|m_{1}m_{2}|,$$
So equality forces $|m_1|=|m_2|=1$, both are circle diffeomorphisms.
Thus, $g_{\infty}\cdot\Gamma_{\infty}=\Hopf$ for some $g_{\infty}\in{\rm G}$.
These preserve convergence in every $C^m$ norm and give the stated conclusion.
\end{proof}

\section{The second variation and infinitesimal M\"obius action at the Hopf Link}\label{sec:Hopf-Hessian}

In this section, we firstly compute the normal directional second variation (the Hessian) of the M\"obius cross energy at the Hopf Link, and we obtain its kernel. Then we identify the kernel with the infinitesimal M\"obius action at the Hopf Link in the normal direction. 

At first, we calculate the second variation of the M\"obius cross energy at Hopf Link. Set
\[
 P_1=\Span\{e_1,e_2\},\qquad P_2=\Span\{e_3,e_4\}.
\]
For the Hopf pair~\eqref{eq:Hopf}, the curve-normal bundles are the constant bundles
$N_xx=P_2$ and $N_yy=P_1$.  Thus normal fields have the form
\[
 U:\Sph^1\to P_2,\qquad V:\Sph^1\to P_1.
\]
For $k=0,1,2$, we set $\mathcal H^k=H^k(\Sph^1,P_2)\oplus H^k(\Sph^1,P_1)$ and $\mathcal H^{-1}=(\mathcal H^1)^\ast$.
In particular,
\begin{equation}\label{eq:H1-norm}
 |W|_{H^1}^2
 =\int_0^{2\pi}(|U|^2+|U'|^2)\dd s
  +\int_0^{2\pi}(|V|^2+|V'|^2)\dd t,
 \qquad W=(U,V).
\end{equation}


\begin{proposition}\label{prop:Hopf-Hessian}
The Hopf link is critical.  For $W=(U,V)\in\mathcal H^2$ its normal Hessian is given by
\begin{align}\label{eq:Hopf-Hessian}
 \delta^{2} E_\Hopf(W)
={}&\pi\int_0^{2\pi}|U'|^2\dd s
 +\pi\int_0^{2\pi}|V'|^2\dd t\\
 &+2\int_0^{2\pi}\int_0^{2\pi}
 \inner{y(t)}{U(s)}\inner{x(s)}{V(t)}\dd s\dd t.\notag
\end{align}
\end{proposition}

\begin{proof}
It suffices first to take smooth fields and use the pointwise spherical exponential variations $x_\tau(s)=\exp_{x(s)}(\tau U(s))$ and $y_\tau(t)=\exp_{y(t)}(\tau V(t))$. Orthogonality of the two
planes gives
\begin{align}
 x_\tau=x+\tau U-\frac{\tau^2}{2}|U|^2x+O(\tau^3),
 \qquad
 y_\tau=y+\tau V-\frac{\tau^2}{2}|V|^2y+O(\tau^3).\notag
\end{align}
Using $x,x'\in P_1$ and $U,U'\in P_2$, and symmetrically for $y,V$, gives
\begin{align}
 |x_\tau'|=1+\frac{\tau^2}{2}(|U'|^2-|U|^2)+O(\tau^3),
 \quad
 |y_\tau'|=1+\frac{\tau^2}{2}(|V'|^2-|V|^2)+O(\tau^3).\notag
\end{align}
Let $ a(s,t)=\inner{U(s)}{y(t)}+\inner{x(s)}{V(t)}.$
For $c_\tau=\inner{x_\tau}{y_\tau}$, orthogonality of $P_1$ and $P_2$ gives
$c_0=0$, $\dot c_0=a$, and $\ddot c_0=2\inner{U}{V}=0$.  Therefore
\begin{align}
 \frac1{|x_\tau-y_\tau|^2}
 =\frac1{2(1-c_\tau)}
 =\frac12+\frac\tau2a+\frac{\tau^2}{2}a^2+O(\tau^3).\notag
\end{align}  
The first-order term is linear in $x(s)$ and $y(t)$. Since both Hopf circles have vanishing mean, $\int_{\mathbb{S}^{1}}x(s)ds=\int_{\mathbb{S}^{1}}y(t)dt=0$, every such term integrates to zero on the parameter torus, hence $\delta E_{H}=0$ and $\Hopf$ is critical.
Twice the coefficient of $\tau^2$ after integration is
\begin{align*}
 \iint a^2\dd s\dd t
 +\pi\int_0^{2\pi}(|U'|^2-|U|^2)\dd s
 +\pi\int_0^{2\pi}(|V'|^2-|V|^2)\dd t.
\end{align*}
Finally,
\begin{equation}\label{eq:moments}
 \int_0^{2\pi}x\otimes x\dd s=\pi\Id_{P_1},
 \qquad
 \int_0^{2\pi}y\otimes y\dd t=\pi\Id_{P_2}.
\end{equation}
Hence the pure squares in $\iint a^2\dd s\dd t$ equal $\pi\int|U|^2$ and $\pi\int|V|^2$ and cancel the negative zeroth-order terms.  The remaining cross term is exactly~\eqref{eq:Hopf-Hessian}. 
\end{proof}

\begin{remark}
    In fact, the first variation extends continuously to \(\mathcal H^1\), while the second variation extends to a bounded bilinear form on \(\mathcal H^1\times\mathcal H^1\); see Lemma \ref{lem:remainder} for details. These extensions will be used in the local analysis.
\end{remark}

Decompose the normal fields into their constant, first-harmonic, and higher-frequency parts:
\begin{align}
 U(s)=U_0+A x(s)+U_{\ge2}(s),\qquad V(t)=V_0+B y(t)+V_{\ge2}(t) \notag
\end{align}
where
$$U_0=\frac1{2\pi}\int_0^{2\pi}U(s)\dd s,\qquad A=\frac{1}{\pi}\int_{0}^{2\pi}U\otimes x\ ds,$$
and $V_{0},B$ are defined analogously. Thus $ A\in\Hom(P_1,P_2),B\in\Hom(P_2,P_1)$, the Fourier expansion of $U_{\ge2}$ contains only frequencies $|n|\ge2$.

If $U\in H^1$, the same orthogonal decomposition holds in $H^1$, because Fourier projection commutes with differentiation.

\begin{proposition}\label{prop:perfect-square}
For every $W\in \mathcal H^1$ one has the exact identity
\begin{equation}\label{eq:perfect-square}
 \delta^{2} E_\Hopf(W)
 =\pi^2|B+A^T|_F^2
 +\pi\int|U_{\ge2}'|^2
 +\pi\int|V_{\ge2}'|^2.
\end{equation}
where $|A|^2_{F}={\rm tr}(A^{T}A)$. Consequently, the nullspace $\KH$ of the bilinear form is
\begin{equation}\label{eq:kernel}
 \KH={}\{(U_0,V_0):U_0\in P_2,\ V_0\in P_1\}\oplus\{(Ax,-A^Ty):A\in\Hom(P_1,P_2)\},
\end{equation}
and $\dim\KH=8$ .
Its $L^2$-orthogonal complement is
\begin{equation}\label{eq:orthogonal-complement}
  \KH^\perp
 =\{(Ax+U_{\ge2},A^Ty+V_{\ge2}):A\in\Hom(P_1,P_2)\}.
\end{equation}
For every $W\in\mathcal H^1\cap\KH^\perp$,
\begin{equation}\label{eq:coercivity}
\delta^{2} E_\Hopf(W)\ge\frac{4\pi}{5}|W|_{H^1}^2.
\end{equation}
\end{proposition}

\begin{proof}
Only the first Fourier modes contribute to the nonlocal term in
\eqref{eq:Hopf-Hessian}. Indeed, put
\[
 I(U,V)=\int_0^{2\pi}\int_0^{2\pi}
 \inner{y(t)}{U(s)}\inner{x(s)}{V(t)}\dd s\dd t.
\]
If $U=U_0$ is constant, integration in $s$ and $\int x(s)\dd s=0$ give
$ I(U_0,V)=0$; similarly $ I(U,V_0)=0$.  If
$U=U_{\ge2}$, then for each fixed $t$ the function
$s\mapsto\inner{y(t)}{U_{\ge2}(s)}$ contains only frequencies $|n|\ge2$, whereas
$s\mapsto\inner{x(s)}{V(t)}$ has frequency $|n|=1$.  Their integral therefore vanishes. Interchanging $s$ and $t$ shows that every term containing $V_{\ge2}$ also vanishes.  Hence
\[
  I(U,V)= I(Ax,By).
\]

Now we compute this remaining term.  Since $\inner{y}{Ax}=\inner{A^Ty}{x}$, the moment identity~\eqref{eq:moments} gives, after first integrating in $s$,
\begin{align*}
 \int_0^{2\pi}\inner{y}{Ax}\inner{x}{By}\dd s
 &=\inner{A^Ty}{\left(\int_0^{2\pi}x\otimes x\dd s\right)By}\\
 &=\pi\inner{A^Ty}{By}
 =\pi\inner{y}{AB y}.
\end{align*}
A second use of~\eqref{eq:moments} yields
\[
  I(Ax,By)
 =\pi\tr\left(AB\int_0^{2\pi}y\otimes y\dd t\right)
 =\pi^2\tr(AB).
\]
If $|A|_F:=\left(\tr(A^TA)\right)^{\frac{1}{2}}$ denotes the Frobenius norm, then
\begin{align*}
 \int_0^{2\pi}|(Ax)'|^2\dd s
 &=\int_0^{2\pi}\inner{Ax'}{Ax'}\dd s
 =\tr\left(A^TA\int_0^{2\pi}x'\otimes x'\dd s\right)
 =\pi|A|_F^2,
\end{align*}
because $\int x'\otimes x'\dd s=\pi\Id_{P_1}$.  The second component gives
$\int|(By)'|^2\dd t=\pi|B|_F^2$ in the same way.  Thus the first-mode contribution is
$$\pi^2(|A|_F^2+|B|_F^2+2\tr(AB))=\pi^2|B+A^T|_F^2,$$
so \eqref{eq:perfect-square} and~\eqref{eq:kernel} are proved.  Orthogonality to the constant kernel removes $U_0,V_0$, while orthogonality to every $(Cx,-C^Ty)$ gives $B=A^T$, this implies
\eqref{eq:orthogonal-complement}.

On the transverse first modes, 
$$\delta^{2} E_\Hopf((Ax,A^Ty))=4\pi^2|A|_F^2\quad \text{and}\quad
|(Ax,A^Ty)|_{H^1}^2=4\pi|A|_F^2.$$
On frequencies $|n|\ge2$, Poincar\'e's inequality gives $\int|f'|^2\ge4\int|f|^2$, and hence
$$\pi\int|f'|^2\ge(4\pi/5)|f|_{H^1}^2.$$
The Fourier pieces are orthogonal, which proves
\eqref{eq:coercivity}.
\end{proof}


The kernel in \eqref{eq:kernel} consists of constant normal fields and
mixed rotational fields. We will construct an explicit family of M\"obius transformations realizing
these velocities.

For $A\in\Hom(P_1,P_2)$, the skew-symmetric matrix
\[
 \begin{pmatrix}0&-A^T\\ A&0\end{pmatrix}
\]
generates a rotation mixing the two orthogonal planes. Its normal trace
along the Hopf pair is $(Ax,-A^Ty)$.
The constant normal fields are generated by gradients of height functions.
For $a=a_1+a_2\in P_1\oplus P_2$, set $f_a(z)=\inner{a}{z}$.
On the unit sphere,
\[
 \nabla f_a(z)=a-\inner{a}{z}z,\qquad
 \nabla^2f_a=-f_a g_0.
\]
Thus $\nabla f_a$ is a conformal vector field. Its normal trace along
the Hopf pair is $(a_2,a_1)$, since the curve-normal planes along $x$
and $y$ are $P_2$ and $P_1$, respectively.

Accordingly, for the uniquely decomposed kernel field
\[
 h=(U_0+Ax,V_0-A^Ty)\in\KH,
\]
define
\begin{equation}\label{eq:kernel-parameters}
 a(h)=U_0+V_0,\qquad
 \Omega(h)=\begin{pmatrix}0&-A^T\\ A&0\end{pmatrix},
\end{equation}
and
\begin{equation}\label{eq:kernel-conformal-field}
 X_h(z)=\Omega(h)z+\nabla f_{a(h)}(z)
       =\Omega(h)z+a(h)-\inner{a(h)}{z}z.
\end{equation}
This assignment is linear in $h$, and $\bigl(P_{P_2}X_h(x),P_{P_1}X_h(y)\bigr)=h.$
In particular, $a(h)$ is chosen to realize the constant part of $h$,
while $\Omega(h)$ realizes its first-harmonic part.

Every M\"obius transformation of $\Sph^3$ is a composition of a rotation in $\operatorname{SO}(4)$ and a standard conformal diffeomorphism $T_b(z)$ of $\Sph^3$ as in the following:
\begin{equation}\label{eq:explicit-T}
 T_b(z)=(1-|b|^2)\frac{z-b}{|z-b|^2}-b,
 \qquad z\in\Sph^3,
\end{equation}
where $|b|<1$.

We see $T_0=\Id$ and $T_{\tau b}$ is a path of sphere diffeomorphisms for $0\le\tau\le1$.
\begin{equation}\label{eq:T-initial-velocity}
 \left.\frac{\mathrm d}{\mathrm d\tau}\right|_{\tau=0}
 T_{\tau b}(z)
 =-2\bigl(b-\inner{b}{z}z\bigr)=-2\nabla f_b(z).
\end{equation}
Consequently, $T_{-\tau a/2}$ has initial velocity $\nabla f_a$.
This fixes the normalization of the conformal parameter in the following
construction.

\begin{proposition}\label{prop:kernel-orbit}
For $h\in\KH$ sufficiently close to zero, define
\begin{equation}\label{eq:explicit-Phi}
 \Phi_h=\exp\bigl(\Omega(h)\bigr)\circ T_{-a(h)/2}.
\end{equation}
Then $\Phi_h\in\Mob^+(\Sph^3)$, the map $(h,z)\mapsto\Phi_h(z)$
is smooth, and $\Phi_0=\Id$. Moreover,
\begin{equation}\label{eq:Phi-differential}
 \left.\frac{\mathrm d}{\mathrm d\tau}\right|_{\tau=0}
 \Phi_{\tau h}(z)=X_h(z),
\end{equation}
so the normal differential of this family along $\Hopf$ is the identity
on $\KH$. Thus $\KH$ is exactly the space of normal infinitesimal
M\"obius motions of the Hopf pair.
\end{proposition}

\begin{proof}
The exponential in \eqref{eq:explicit-Phi} is the ordinary matrix
exponential and belongs to $\operatorname{SO}(4)$.
Linearity of $a(h)$ and $\Omega(h)$, together with
\eqref{eq:T-initial-velocity}, gives
\begin{equation}
 \Phi_{\tau h}(z)
 =z+\tau\bigl(\Omega(h)z+a(h)-\inner{a(h)}{z}z\bigr)
   +O(\tau^2).
\end{equation}
By combining \eqref{eq:kernel-conformal-field}, we have \eqref{eq:Phi-differential}, i.e. every kernel field is realized as a normal M\"obius velocity.

Conversely, let us eliminate the tangential velocity of a M\"obius motion
by reparametrization, then the energy invariance implies zero second variation in its remaining normal velocity, which therefore belongs to $\KH$.
\end{proof}


We now derive the linearized operator from the quadratic form.  For
$W=(U,V)$ and $\widetilde W=(\widetilde U,\widetilde V)$ define the polarization
\[
 B_\Hopf(W,\widetilde W)
 :=\frac12\Bigl(\delta^{2} E_\Hopf(W+\widetilde W)-\delta^{2} E_\Hopf(W)-\delta^{2} E_\Hopf(\widetilde W)\Bigr).
\]
Substitution of~\eqref{eq:Hopf-Hessian} and integrating by parts gives
\[
 B_\Hopf(W,\widetilde W)
 =\int\inner{(\LH W)_1}{\widetilde U}\dd s
 +\int\inner{(\LH W)_2}{\widetilde V}\dd t,
\]
where the self-adjoint operator is
\begin{equation}\label{eq:LH}
\begin{split}
    (\LH W)_1(s)&=-\pi U''(s)+\int_0^{2\pi}y(t)\inner{x(s)}{V(t)}\dd t,\\
 (\LH W)_2(t)&=-\pi V''(t)+\int_0^{2\pi}x(s)\inner{y(t)}{U(s)}\dd s.
\end{split}
\end{equation}
On $\mathcal H^0$ with domain $\mathcal H^2$, $\LH$ is a self-adjoint
operator. 
It also admits the weak realization
$$\LH:\mathcal H^1\to\mathcal H^{-1}$$ 
where the second derivatives in \eqref{eq:LH} are understood distributionally. Moreover,
\[
 \langle \LH W,\widetilde W\rangle=B_\Hopf(W,\widetilde W).
\]
Its kernel is $\KH$.
Since \(\mathcal K_H\) is finite-dimensional and consists of smooth fields, the \(L^2\)-orthogonal projections 
\begin{align}\label{eq:dual-projection}
    P_{\mathcal K_H},\quad P_{\perp}=I-P_{\mathcal K_H}
\end{align}
extend boundedly to \(\mathcal H^1\) and \(\mathcal H^{-1}\).

\section{Nonlinear estimate and renormalization lemma}\label{sec:local}
In this section, we derive a nonlinear estimate for the first variation near the Hopf link. Motivated by the renormalization lemma of Mondino and Nguyen \cite{MN14}, we then construct a local Möbius gauge in which any sufficiently nearby link is represented as a normal graph over \(\Hopf\) with graph field orthogonal to \(\KH\).


Fix $0<r_0<\pi/4$.  The normal exponential maps
\begin{align}
 \Psi_1(s,\xi)&=\exp_{x(s)}\xi,
 & (s,\xi)&\in\Sph^1\times B_{r_0}(0;P_2),\label{eq:Psi1}\\
 \Psi_2(t,\eta)&=\exp_{y(t)}\eta,
 & (t,\eta)&\in\Sph^1\times B_{r_0}(0;P_1),\label{eq:Psi2}
\end{align}
are diffeomorphisms onto disjoint tubular neighborhoods of the Hopf circles. For a small field $v=(U,V)$, set
\begin{align}
 \Graph_\Hopf(v)=(X_U,Y_V),\qquad
 X_U(s)=\Psi_1(s,U(s)),\quad Y_V(t)=\Psi_2(t,V(t)),\notag
\end{align}
and define
\begin{align}
 \mathcal E(v)=E(\Graph_\Hopf(v)),\qquad
 \langle\EL(v),\zeta\rangle=\delta\mathcal E(v)[\zeta].\notag
\end{align}
Here $v\in\mathcal H^2$ has small $C^1$ norm, and $\zeta\in\mathcal H^2$.
The notation $\delta\EL(v)[\zeta]$ denotes its linearization in the direction $\zeta$ and satisfies
\begin{align}
 \langle\delta\EL(v)[\zeta],\eta\rangle
 =\delta^2\mathcal E(v)[\zeta,\eta].\notag
\end{align}

The graph variation in the first component is
$$\partial_\xi\Psi_1(s,U(s))[\zeta_1(s)].$$
Its projection to the normal plane of the current curve is an isomorphism at $v=0$, and remains so for small
$|v|_{C^1}$.  The second component has the same property.  For $|v|_{C^1}$ small, the normal projection of the graph variations onto the normal bundles of the current curves is pointwise invertible. Hence, \Cref{prop:first-variation} gives
\begin{equation}\label{eq:EL-equivalence}
 \EL(v)=0\quad\Longleftrightarrow\quad
 \Graph_\Hopf(v)\text{ is weakly critical}.
\end{equation}

\begin{lemma}\label{lem:remainder}
There are $C,r_1>0$ such that, for $v,w\in\mathcal H^2$ with
$|v|_{C^1},|w|_{C^1}<r_1$, the linearization $\delta\EL(v)$ extends to a uniformly bounded map $\mathcal H^1\to\mathcal H^{-1}$ and
\begin{equation}\label{eq:Hessian-Lipschitz}
 |\delta\EL(v)-\delta\EL(w)|_{\Lin(\mathcal H^1,\mathcal H^{-1})}
 \le C|v-w|_{C^1}.
\end{equation}
Here $\Lin(X,Y)$ is the space of bounded linear maps with its operator norm.
Moreover, $\EL(0)=0$, $\delta\EL(0)=\LH$, and
\begin{equation}\label{eq:remainder}
 |\EL(v)-\LH v|_{H^{-1}}
 \le C|v|_{C^1}|v|_{H^1}.
\end{equation}
\end{lemma}
\begin{proof}
    The chain rule gives
\[
 X_U'=\partial_s\Psi_1(s,U)+\partial_U\Psi_1(s,U)[U'],\qquad
 Y_V'=\partial_t\Psi_2(t,V)+\partial_V\Psi_2(t,V)[V'] .
\]
Hence
\begin{equation}\label{eq:graph-Lagrangian}
 \mathcal E(v)=\iint f(s,t,U(s),U'(s),V(t),V'(t))\dd s\dd t,
\end{equation}
where
\begin{align}
 f(s,t,U,U',V,V')
 =\frac{|\partial_s\Psi_1(s,U)+\partial_U\Psi_1(s,U)[U'] |\,
          |\partial_t\Psi_2(t,V)+\partial_V\Psi_2(t,V)[V'] |}
        {|\Psi_1(s,U)-\Psi_2(t,V)|^2}.\notag
\end{align}
Let \(\mathbf q=(U,U',V,V')\). For $|v|_{C^1}<r_{1}$, the speeds and the separation of the two components are uniformly bounded away from zero. Hence \(D_\mathbf q^k f\) is uniformly bounded for \(k\le3\).

To display the differentiation, abbreviate the tuple of values and first
derivatives by
\[
 \mathbf z_v(s,t)=(U(s),U'(s),V(t),V'(t)).
\]
We define \(z_\zeta,z_\eta\) analogously for $\zeta,\eta\in \mathcal H^{2}$.
Twice differentiating the finite-dimensional integrand gives
\begin{equation}\label{eq:coefficient-Hessian}
 \langle\delta\EL(v)[\zeta],\eta\rangle
 =\iint \partial_{\mathbf q}^2f(s,t,\mathbf z_v)
          [\mathbf z_\zeta,\mathbf z_\eta]\dd s\dd t.
\end{equation}
Since \(\partial_{\mathbf q}^2f\) is uniformly bounded, each term in
$ D_{\mathbf q}^2f(s,t,z_v)[z_\zeta,z_\eta] $ is bounded by a constant times a product $a\,c$, where $a$ is one of
$$ |\zeta_1(s)|,\quad |\zeta_1'(s)|,\quad |\zeta_2(t)|,\quad |\zeta_2'(t)|, $$
and $c$ is the corresponding quantity associated with $\eta$. Hence, by Cauchy-Schwarz inequality
$$  \left| D^2\mathcal E(v)[\zeta,\eta] \right| \le C\iint_{\mathbb T^2}|z_\zeta(s,t)|\,|z_\eta(s,t)|\,ds\,dt\le C|z_\zeta|_{L^2(\mathbb T^2)} |z_\eta|_{L^2(\mathbb T^2)}.  $$
Since
$$ |z_\zeta|_{L^2(\mathbb T^2)}^2 = 2\pi|\zeta|_{H^1}^2, \qquad |z_\eta|_{L^2(\mathbb T^2)}^2 = 2\pi|\eta|_{H^1}^2, $$
we obtain
$ \left| D^2\mathcal E(v)[\zeta,\eta] \right| \le C|\zeta|_{H^1}|\eta|_{H^1}.  $
Thus \(D^2\mathcal E(v)\) extends uniquely to a bounded bilinear form on \(\mathcal H^1\times\mathcal H^1\).

Finally, the bound for the third derivatives of $f$ and the mean-value theorem
give
\[
 \sup_{s,t}\bigl|\partial_{\mathbf q}^2f(s,t,\mathbf z_v)
                -\partial_{\mathbf q}^2f(s,t,\mathbf z_w)\bigr|
 \le C|v-w|_{C^1}.
\]
Applying the same product estimates to this difference proves
\eqref{eq:Hessian-Lipschitz}.  At $v=0$ the chart Hessian is the normal Hopf Hessian, so
$\delta\EL(0)=\LH$.  Integration along $\tau v$ now gives
\[
 \EL(v)-\LH v
 =\int_0^1\bigl(\delta\EL(\tau v)-\delta\EL(0)\bigr)[v]\dd\tau,
\]
and \eqref{eq:Hessian-Lipschitz} proves \eqref{eq:remainder}.
\end{proof}

Write $\Psi_i^{-1}=(\pi_i,\nu_i)$.  For $Z=(Z_1,Z_2)$ close to $\Hopf$ in
$C^1$, let
\[
 \theta_i=\pi_i\circ Z_i,\qquad \eta_i=\nu_i\circ Z_i.
\]
Then $\theta_i$ is a degree-one circle diffeomorphism close to the identity.
The field of its normal-graph representative is
\begin{equation}\label{eq:graph-extraction}
 \ngraph(Z):=(\eta_1\circ\theta_1^{-1},\eta_2\circ\theta_2^{-1}).
\end{equation}

\begin{lemma}\label{lem:graph-extraction}
On a sufficiently small $C^1$ neighborhood of $\Hopf$, the functions
$\theta_i,\eta_i$ satisfy $\theta_i'>1/2$, after taking lifts close to the identity.
They define degree-one circle diffeomorphisms and
\[
 u_i=\eta_i\circ\theta_i^{-1},\qquad
 Z_i(r)=\Psi_i\bigl(\theta_i(r),u_i(\theta_i(r))\bigr).
\]
The assignment $Z\mapsto\ngraph(Z)=(u_1,u_2)$ is continuous in $C^1$, with
$\ngraph(\Hopf)=0$.  If $Z\in H^2$, then $\theta_i$, $\theta_i^{-1}$, and
$u_i$ belong to $H^2$.
\end{lemma}

\begin{proof}
At $Z=\Hopf$, $\theta_i(r)=r$ and $\eta_i=0$.
Smoothness of the tubular coordinates gives the lower derivative bound and the
identity $\theta_i(r+2\pi)=\theta_i(r)+2\pi$ on a small $C^1$ neighborhood.
Thus each $\theta_i$ is a degree-one diffeomorphism.

For completeness, suppose $Z^{(n)}\to Z$ in $C^1$.
Smooth pointwise composition gives $\theta_i^{(n)}\to\theta_i$ and
$\eta_i^{(n)}\to\eta_i$ in $C^1$.
The inverse lifts are uniformly Lipschitz, with constant at most $2$, so their
uniform convergence follows from that of the original lifts.  The formula
\[
 ((\theta_i^{(n)})^{-1})'
 =\frac{1}{(\theta_i^{(n)})'\circ(\theta_i^{(n)})^{-1}}
\]
and uniform continuity of the limiting derivative give convergence in $C^1$.
Applying the chain rule to $u_i^{(n)}$ proves $u_i^{(n)}\to u_i$ in $C^1$.

If $Z\in H^2$, the ordinary Sobolev chain rule gives
$\theta_i,\eta_i\in H^2$.  Write $\psi_i=\theta_i^{-1}$.  Almost everywhere,
\[
 \psi_i''=-\frac{\theta_i''\circ\psi_i}{(\theta_i'\circ\psi_i)^3},
 \qquad
 u_i''=(\eta_i''\circ\psi_i)(\psi_i')^2
           +(\eta_i'\circ\psi_i)\psi_i''.
\]
The derivative bounds for $\theta_i$ and change of variables imply
$\psi_i''\in L^2$.  The first term in $u_i''$ lies in $L^2$ by the same change
of variables, and the second does because $\eta_i'$ and $\psi_i'$ are bounded.
This proves the $H^2$ assertions.
\end{proof}

Let \(P_{\mathcal K_H}\) denote the \(L^2\)-orthogonal projection onto the kernel of the Hopf Hessian. We define $M(Z)$ by
$$ M(Z):=P_{\mathcal K_H}\operatorname{ng}_H(Z). $$
For $h\in\KH$ sufficiently close to zero, let $\Phi_h$ be the
M\"obius transformation defined in \eqref{eq:explicit-Phi}.
Set
\[
 \mathcal X^1=C^1(\Sph^1,P_2)\oplus C^1(\Sph^1,P_1),
\]
and define, on sufficiently small zero neighborhoods in
$\KH\times\mathcal X^1$,
\[
 \mathcal F(h,w)
 =M\bigl(\Phi_h(\Graph_\Hopf(w))\bigr).
\]

Based on the above argument, the following lemma, motivated by the renormalization lemma in \cite{MN14}, provides a local M\"obius gauge near $\Hopf$: after a small M\"obius transformation and reparametrization, a nearby link can be written as a normal graph over $\Hopf$, with its graph field orthogonal to $\KH$.
\begin{lemma}\label{lem:slice}
For every sufficiently small $w\in\mathcal X^1$, there is a unique sufficiently
small $h(w)\in\KH$ such that
\begin{align}\label{5.11}
     v=\ngraph\bigl(\Phi_{h(w)}(\Graph_\Hopf(w))\bigr)
 \quad\text{satisfies}\quad P_{\KH}v=0.
\end{align}
Moreover, \(h(w)\) depends smoothly on \(w\), \(v\to0\) in \(C^1\) as \(w\to0\), and \(H^2\)-regularity is preserved.
\end{lemma}
\begin{proof}
    Let \(k_1,\ldots,k_8\) be a smooth \(L^2\)-orthonormal basis of \(\mathcal K_H\) with $k_a=(k_{a,1},k_{a,2})$. Define
    \begin{align}\label{5.12}
        m_a(Z)=\sum_{i=1}^2\int_0^{2\pi} \inner{\eta_i(r)}{k_{a,i}(\theta_i(r))}\,\theta_i'(r)\dd r.
    \end{align}
From the definition of $ \Psi_{i}^{-1}$, changing variables $q=\theta_i(r)$ in the formula for $m_a$ yields
\[
 m_a(Z)=\sum_{i=1}^2\int_0^{2\pi}
       \inner{\eta_i(\theta_i^{-1}(q))}{k_{a,i}(q)}\dd q
       =\langle\ngraph(Z),k_a\rangle_{L^2}.
\]
Thus, we have
$$ M(Z) = \sum_{a=1}^8m_a(Z)k_a.$$
Indeed, we can define
$$ f_{a,i}(z,p) := \langle\nu_i(z),k_{a,i}(\pi_i(z))\rangle\,d\pi_i|_z[p], $$
which is smooth in \((z,p)\).  
Thus 
$$m_{a}(Z)=\sum_{i=1}^2\int_0^{2\pi}
       f_{a,i}(Z_{i}(r),Z_{i}'(r))\dd r$$
is smooth in the \(C^1\) topology, which means $M$ is smooth in the \(C^1\) topology.

For a \(C^1\) sphere-tangent field \(\zeta\) along \(H\), we now compute \(d M(H)[\zeta]\). At \(Z=H\),
$$ \eta_i=0,\qquad \theta_i(r)=r,\qquad \theta_i'=1. $$
In differentiating \eqref{5.12}, the terms differentiating \(k_{a,i}(\theta_i)\) and \(\theta_i'\) therefore vanish. The remaining term is
$$d m_a(H)[\zeta] = \sum_{i=1}^2 \int_0^{2\pi} \left\langle d\nu_i|_{H_i(r)}[\zeta_i(r)],k_{a,i}(r) \right\rangle dr. $$
The differentials of the inverse tubular coordinates give
$$ d M(H)[\zeta] = P_{\mathcal K_H}(P_{P_2}\zeta_1,P_{P_1}\zeta_2). $$
The graph map and the finite-dimensional M\"obius action are smooth in the \(C^1\) topology. Hence \(\mathcal F\) is smooth near \((0,0)\), and \(\mathcal F(0,0)=0\). By \Cref{prop:kernel-orbit}, for every $k\in\KH$,
\[
 \mathrm d_h\mathcal F(0,0)[k]
 =\mathrm dM_\Hopf[(X_k\circ x,X_k\circ y)]
 =P_{\KH}k
 =k.
\]
Thus $\mathrm d_h\mathcal F(0,0)$ is the identity on $\KH$.
The implicit-function theorem gives a unique smooth small
parameter $h=h(w)$, with $h(0)=0$ and
$\mathcal F(h(w),w)=0$. This proves $\eqref{5.11}$ and the asserted local uniqueness of the parameter. The remaining continuity and regularity statements follow from \Cref{lem:graph-extraction}.
\end{proof}

\section{Proof of the critical-value gap}\label{sec:gap}
In this section, we establish a local gap theorem near the Hopf orbit. Combining this local rigidity result with the compactness theorem in \Cref{prop:threshold}, we then prove the main theorem.
\begin{theorem}\label{thm:local}
There is a $C^{1}$ neighborhood of the Hopf orbit such that
every weakly critical pair in this neighborhood belongs to ${\rm G}\cdot\Hopf$.
\end{theorem}
\begin{proof}
By \Cref{lem:graph-extraction}, a pair $\Gamma\in\sA$ sufficiently close to
$\Hopf$ in $C^1$ can be separately reparametrized as $\Graph_\Hopf(w)$, with
$w\in\mathcal H^2$ small in $C^1$.

Apply \Cref{lem:slice} to obtain a M\"obius-equivalent graph
$\Graph_\Hopf(v)$ with $v\perp_{L^2}\KH$.
The continuity statements in the two lemmas allow us to shrink the original
neighborhood so that $|v|_{C^1}<r_{\ast}$, where
$r_{\ast}<r_1$ and $Cr_{\ast}<4\pi/5$ for the constants in \Cref{lem:remainder}.

Weak criticality is preserved by both operations from
\Cref{lem:constant-speed}. We write 
$$\mathcal R(v)=\EL(v)-\LH v,$$
then Lemma \ref{lem:remainder} implies $\LH v=-\mathcal R(v)$.
Pairing with $v$ gives
\begin{align}
 \frac{4\pi}{5}|v|_{H^1}^2
 \le \delta^{2} E_\Hopf(v)
 =-\langle\mathcal R(v),v\rangle
 \le Cr_{\ast}|v|_{H^1}^2.
\end{align}
Since $Cr_*<4\pi/5$, one has $v=0$.  Undoing the transformations proves
$\Gamma\in{\rm G}\cdot\Hopf$.
Finally, each fixed element of ${\rm G}$ acts by a $C^1$ homeomorphism and preserves
weak criticality, which proves the statement on the union of translates.
\end{proof}


\begin{proof}[Proof of \Cref{thm:main}]
Suppose that no positive gap exists.  Then, there is a sequence of non-split,
weakly critical pairs $\Gamma_n\in\sA\setminus({\rm G}\cdot\Hopf)$ such that
\[
 2\pi^2\le E(\Gamma_n)\le2\pi^2+\frac1n.
\]
By \Cref{prop:threshold}, a subsequence converges smoothly to $\Hopf$ after
common M\"obius transformations and separate reparametrizations.  For large $n$
the resulting pairs lie in the neighborhood of \Cref{thm:local}, and hence
belong to ${\rm G}\cdot\Hopf$.  Invariance of this orbit gives a
contradiction.
\end{proof}

\end{document}